\documentclass[10pt, letterpaper]{article}
\usepackage{hyperref}
\usepackage{amsmath,amssymb, amsopn, amsthm}
\usepackage{enumerate, framed, xcolor}

\usepackage{mathtools}
\usepackage{enumerate}

\usepackage{tikz-cd}
\usetikzlibrary{graphs, graphs.standard}
\usepackage{enumitem}

\usepackage{titlesec}
\titleformat{\section}
  {\normalfont\large\bfseries}
  {\thesection}{1em}{}

\titlespacing*{\section}
  {0pt}{3ex plus 1ex minus .2ex}{2ex}
\titleformat{\subsection}
  {\normalfont\normalsize\bfseries}
  {\thesubsection}{1em}{}

\titlespacing*{\subsection}
  {0pt}{2ex plus .5ex}{1.5ex}

\theoremstyle{definition}
\newtheorem{theorem}{Theorem}[section]
\newtheorem{lemma}[theorem]{Lemma}

\newtheorem{corollary}[theorem]{Corollary}

\newtheorem{definition}[theorem]{Definition}

\renewenvironment{proof}[1][\unskip]{%
\par
\noindent
\textbf{Proof #1.}
\noindent}
{\hfill$\blacksquare$

\bigskip}

\usepackage{etoolbox}

\newcommand{\defqedsymbol}{\hfill\ensuremath{\square}} 

\AtBeginEnvironment{definition}{\pushQED{\defqedsymbol}}
\AtEndEnvironment{definition}{\popQED}

\let\<=\langle
\let\>=\rangle

\usepackage[colorinlistoftodos]{todonotes}

\title{Uncountably many conditionally inaccessible decisions exist in every finite probability space}
\author{Zal\'an Gyenis\thanks{Jagiellonian University, Krak\'ow}\and Mikl\'os R\'edei\thanks{London School of Economics and Political Science, London}\and Leszek Wro\'nski${}^{*}$}

\date{\today}

\newcommand{\E}{\mathbb E}

\begin{document}

\maketitle

\begin{abstract}
In a recent paper \cite{Redei-Jing2026} the notion of conditional $p$-inaccessibility of a decision based on utility maximization was defined and examples of conditionally $p$-inaccessible decisions were given. 
The conditional inaccessibility of a decision based on maximizing utility calculated by a probability measure $p^*$ expresses that the decision cannot be obtained if the expectation values of the utility functions are calculated using the (Jeffrey) conditional probability measure obtained by conditioning $p$ on partial evidence about the probability $p^*$ that determines the decision. The paper \cite{Redei-Jing2026} conjectured that conditionally $p$-inaccessible decisions exist in some probability spaces having arbitrary large finite number of elementary events. In this paper we prove that for any $p$ in any finite probability space there exist an uncountable number of probability measures $p^*$ for each of which there exist an uncountable number of pairs of utility functions that represent conditionally $p$-inaccessible decisions. If $p^*$ is an objective probability determining objectively good decisions and $p$ is the subjective probability determining a rational decision of a decision making Agent, the result says that there is an enormous number of decision situations in which the Agent's subjective probability prohibits the Agent's informed rational decision to be objectively good.      
\end{abstract}

\section{The main claims}
The notion of a \emph{conditionally $p$-inaccessible decision} in the context of decision theory based on utility maximization was introduced by R\'edei and Jing in their recent paper \cite{Redei-Jing2026}. 

In decision theory based on utility maximization, actions of an Agent are represented by utility functions and the Agent decides between two actions on the basis of the order of the expectation values of the utility functions calculated by the Agent using a probability measure $p^*$: the Agent prefers the action with higher $p^*$-expected utility. The conditional $p$-inaccessibility of such a decision expresses that the decision cannot be obtained if the expectation values of the utility functions are calculated using the conditional probability measure obtained by conditioning the probability measure $p$ on partial evidence about the probability $p^*$ that determines the decision (Definition 3.2 in \cite{Redei-Jing2026}). It is shown in \cite{Redei-Jing2026} that there exist conditionally $p$-inaccessible decisions in the context of \emph{some} probability spaces having a finite number of elementary events; furthermore, it is proved in \cite{Redei-Jing2026} that if there is a conditionally $p$-inaccessible decision in a finite probability space then there exist uncountably many $p$-inaccessible decisions in that probability space (Proposition 6.1 in \cite{Redei-Jing2026}). It is conjectured in \cite{Redei-Jing2026} that this situation is typical in the sense that for \emph{any} natural number $n\geq 3$ conditionally inaccessible decisions exist in \emph{some} probability spaces with $n$ number of elementary events (Conjecture 7.3 in \cite{Redei-Jing2026}).

In this paper we prove a theorem that is stronger than the conjecture stated above: We show that for \emph{any} probability measure $p$ in \emph{any} probability space having $n\geq 3$ number of elementary events there exist uncountably many probability measures $p^*$ for each of which there exist uncountably many pairs of utility functions that represent conditionally $p$-inaccessible decisions (Theorem \ref{thm:main1}). 

The significance for decision theory of the abundance of conditionally $p$-inaccessible decisions transpires if one interprets $p^*$ as \emph{objective} probability that determines an \emph{objectively good} decision and if $p$ is the \emph{subjective} probability of a decision making Agent: If an objectively good decision based on utility maximization using $p^*$ is conditionally $p$-inaccessible, then the Agent's subjective probability \emph{prohibits} the Agent's \emph{informed rational decision} to be objectively good. ``\emph{Informed rational decision}'' means here that the Agent makes the decision on the basis of calculating the expectation values of the utility functions using the subjective probability $p$ upgraded by conditionalizing it on true but incomplete evidence about $p^*$. 

The paper \cite{Redei-Jing2026} also introduces a discrete, finite-range measure of degree of conditional $p$-inaccessibility of a decision (Definition 7.1 in \cite{Redei-Jing2026}). The measure reflects how far a decision is from being conditionally $p$-inaccessible; thereby the measure reflects how suitable a subjective probability $p$ is to make objectively good decisions on the basis of upgrading $p$ by conditioning on information about $p^*$ and using the upgraded probability to calculate expectation values of utility functions. Examples are given in \cite{Redei-Jing2026} that display different degrees of conditionally $p$-inaccessible decisions, and it is conjectured in \cite{Redei-Jing2026} that given any decision context with a fixed pair of utility functions and a fixed $p^*$, for any logically possible degree $k$ of conditional inaccessibility there exists priors $p_k$ such that the decision is conditionally $p_k$-inaccessible (Conjecture 7.2 in \cite{Redei-Jing2026}). We investigate here the properties of the degree of conditional $p$-inaccessibility of a decision for a general pair $(p^*,p)$ and show that this conjecture is true (Theorem \ref{thm:main2}). 

\section{Definition of conditionally $p$-inaccessible decisions\label{sec:defs}}

Throughout the paper $(X,\mathcal{S})$ denotes a measurable space with $X$ being the finite set $X=\{1,\dots,n\}$ with $n\ge 3$ and $\mathcal{S}$ being the Boolean algebra of the power set of $X$. For a probability measure $q$ on $\mathcal{S}$ we write $q(i)$ instead of $q(\{i\})$ and $\E_q[f]\doteq \sum_{i=1}^n f(i)q(i)$ denotes the expectation value of the random variable $f\colon X\to \mathbb{R}$ with respect to $q$. 
 
A partition of $X$ is denoted by $\Pi = \{B_1,\dots,B_m\}$, where the blocks $B_j$ are nonempty, disjoint, and we have $\bigcup_j B_j=X$. We call $\Pi$ \emph{proper non-trivial} if $2\le m\le n-1$. The set of all proper non-trivial partitions is denoted by $\mathfrak P$. 

The next definition recalls the concept of conditioning in the context of probability spaces having a finite number of elementary events. The definition is a special case of conditioning with respect to Boolean subalgebras (see e.g. Chapter 6. in \cite{Billingsley1995}, or any of \cite{Williams1991}, \cite{Ash-Doleans-Dade1999}, \cite{Rosenthal2006} for the mathematical theory of conditioning with respect to Boolean algebras; more generally with respect to $\sigma$-fields). In the philosophical literature conditioning with respect to Boolean algebras generated by countable partitions is called "Jeffrey conditioning" \cite{Jeffrey1983}, \cite{Jeffrey1992} (the terminology "probability kinematics" also is used to refer to Jeffrey conditioning, see \cite{Diaconis-Zabell1982}).    
\begin{definition}
Given $(X,\mathcal{S})$ and probability measures $p^*$ and $p$ on $\mathcal{S}$, for any partition $\Pi$ the Jeffrey posterior $q_{\Pi}$ of $p$ determined by $\Pi$ is
\begin{equation}\label{eq:Jeffrey-mixture}
q_{\Pi}(i)\;=\; \sum_{B\in\Pi} p^*(B)\, p(i\mid B),
\end{equation}
where $p(i\mid B)=p(i)/p(B)$ for $i\in B$. Thus, the distribution $q_\Pi$ on $X$ is given by
\[
q_\Pi(i)=p^*(B)\,p(i\mid B)\qquad\text{for }i\in B\in\Pi.
\]
\end{definition}
Note that $q_{\Pi}$ is an extension to $\mathcal{S}$ of the restriction of $p^*$ to the Boolean subalgebra $\mathcal{A}$ of $\mathcal{S}$ generated by the partition $\Pi$. The Jeffrey posterior $q_{\Pi}$ was denoted by $p^*_{p,\mathcal{A}}$ in \cite{Redei-Jing2026}. To keep notation simple, instead of $p^*_{p,\mathcal{A}}$, we use $q_{\Pi}$ throughout the paper, leaving out reference to $p^*$ and $p$, which are known from the context. 

\begin{definition}
Given a probability space $(X,\mathcal{S}, q)$ and two random variables $f_1,f_2\colon X\to\mathbb{R}$, we call
\begin{equation}
\big\langle (X,\mathcal{S},q),f_1,f_2 \big\rangle
\end{equation}
a \emph{decision context} and the inequality 
\begin{equation}\label{eq:decision-1}
\E_q[f_1]  >\E_q[f_2]
\end{equation}
a \emph{decision}. We call the decision context \emph{trivial} if $f_1(x)>f_2(x)$ for all $x$ in $X$. To avoid trivial decision situations, in what follows we assume that all decision contexts are non-trivial. In a non-trivial decision context the decision (\ref{eq:decision-1}) can be equivalently written as 
\begin{equation}\label{eq:decision-2}
\E_{p^*}[d]>0
\end{equation}
with the non-strictly positive function $d=f_1-f_2$. 
\end{definition}
The decision theoretic interpretation of the elements in the decision context $\big\langle (X,\mathcal{S},{p}^*),f_1,f_2 \big\rangle$ is the following: The random variables $f_1$ and $f_2$ represent possible actions of a decision making Agent; the numbers $f_1(x)$ and $f_2(x)$ represent the values of the actions $f_1$ and $f_2$ from the perspective of the Agent if the state $x$ of the world obtains -- the larger $f_1(x)$ and $f_2(x)$, the more attractive the actions $f_1$ and $f_2$ are from the Agent's perspective.
The probability measure $q$ in (\ref{eq:decision-1})-(\ref{eq:decision-2}) can be interpreted in two ways: 
\begin{itemize}
    \item[(i)] $q$ can be viewed as an objective probability (e.g. describing relative frequencies). In this case the decision (\ref{eq:decision-1})-(\ref{eq:decision-2}) is \emph{objectively good} because it expresses that action $f_1$ has a higher average value than action $f_2$ does. If $q$ is interpreted as objective probability, we use the notation $p^*$ to refer to it.   
    \item[(ii)] $q$ can be viewed as a subjective probability (also called "credence") expressing degrees of beliefs or expectations of an Agent. In this case the decision (\ref{eq:decision-1})-(\ref{eq:decision-2}) is \emph{rational} because it is in harmony with the Agent's expectations. When $q$ is interpreted this way, we use the notation $p$ to refer to it. 
\end{itemize}
See the monographs  \cite{Fishburn1970}, \cite{Kreps2019}, \cite{Gilboa2009}, \cite{Bradley2017} for the details and \cite{BriggsSEP2023}, \cite{BuchakSEP2022} for compact reviews of the main ideas of utility theory.

The next definition was given in \cite{Redei-Jing2026}, it specifies the central notion of the paper: the concept of a conditionally $p$-inaccessible decision.
\begin{definition}
Given a decision context $\big\langle (X,\mathcal{S},{p}^*),f_1,f_2\big\rangle$ with an objective probability $p^*$, and given a probability measure $p$ on $\mathcal{S}$ representing credence, the decision 
\begin{equation}\label{eq:decision-3}
\E_{p^*}[d]>0
\end{equation} 
with $d=f_1,-f_2$ is called 
\begin{itemize}
\item \emph{conditionally} \emph{$p$-inaccessible} if
\begin{equation}\label{eq:cond-inacc}
\E_{q_\Pi}[d]\le 0 \ \text{ for all }  \Pi\in\mathfrak{P}.
\end{equation}
and 
\item \emph{conditionally strongly} $p$-\emph{inaccessible} if 
\begin{equation}
\E_{q_\Pi}[d]<0 \ \text{ for all }  \Pi\in\mathfrak{P}.
\end{equation}
\end{itemize}
\end{definition}
\bigskip
The interpretation of conditional $p$-inaccessibility of a decision is that
\begin{quote}
    ... the prior $p$ of the Agent makes it impossible for the Agent to reach the objectively good decision if the Agent follows the conditioning strategy: [...] if the Agent calculates the expectation values of the utility functions using probabilities inferred via conditionalizing his prior $p$ on partial information about the objective probability, then either the Agent cannot make a decision between $f_1$ and $f_2$  (because the Agent is indifferent between $f_1$ and $f_2$), or the decision the Agent makes will be objectively wrong -- \emph{no matter what partial information the Agent has about the objective probability}. \cite{Redei-Jing2026}
\end{quote}
The difference between conditional $p$-inaccessibility and conditional \emph{strong} $p$-inaccessibility is that if a decision is conditionally strongly $p$-inaccessible then the Agent having prior $p$ is never neutral about which action to take but the Agent's decision will always be wrong if it is based on calculating the expectation values of the utility functions using the updated prior $q_\Pi$. 

It was shown in \cite{Redei-Jing2026} that if there exist conditionally $p$-inaccessible decisions, then there exist conditionally \emph{strongly} $p$-inaccessible decisions as well; hence, from the perspective of the question of existence of $p$-inaccessible decisions, the difference between strong and "regular" $p$-inaccessibility is irrelevant.   

Note that the definition of conditional $p$-inaccessibility does \emph{not} include the condition that the Agent's decision based on the \emph{unconditioned} $p$ is also objectively wrong:  the definition of conditional $p$-inaccessibility does \emph{not} require $\E_p(d)\leq 0$. This leads to the question of whether it can happen that the decision made by the Agent based on utility calculations using $p$ is objectively good (i.e. $\E_p(d)> 0$) but at the same time it is conditionally $p$-inaccessible. If this could happen, it would be very surprising: it would mean that learning relevant objective truth and taking it into account by conditioning on it, whereby the updated prior gets closer to the objective probability, would lead to a wrong decision -- whereas without the learning the decision based on non-updated prior was objectively good already. We show here that this cannot happen: an informed rational decision (i.e. a decision based on calculating the expectation values using $q_{\Pi}$) cannot be worse than a less informed one (based on calculating expectation values using $p$). This is formulated in Theorem \ref{thm:main2} below. 
\begin{theorem}\label{thm:main2}
	Let $\big\langle (X,\mathcal{S},{p}^*),f_1,f_2\big\rangle$ be a decision context with an objective probability $p^*$, and decision 
\begin{equation}
\E_{p^*}[d]>0
\end{equation} 
with $d=f_1,-f_2$. Let $p$ be  probability measure $p$ on $\mathcal{S}$ representing credence. Suppose that 
	\[
	\E_{q_\Pi}[d]\le 0\ \ \text{ for all }  \Pi\in\mathfrak{P}\,.
	\]
	Then $\E_p[d]<0$.	
\end{theorem}
We prove this theorem in the Appendix. 

\section{Existence of an uncountable number of conditionally inaccessible decisions in all finite probability spaces}
The existence of conditionally $p$-inaccessible decisions is non-trivial: The number of inequalities in eq. (\ref{eq:cond-inacc}) that express conditional inaccessibility is equal to the cardinality of the set of all proper non-trial partitions of $X$, which grows exponentially as the number of elementary events in $X$ gets larger. Thus, checking numerically the conditional inaccessibility of decisions in specific decision contexts becomes an intractable problem as the number of elements in $X$ grows. But there \emph{are} conditionally $p$-inaccessinble decisions: The paper \cite{Redei-Jing2026} gives some examples of conditionally $p$-inaccessible decisions in probability spaces having 3 and 4 elementary events -- in the latter case the number of non-trivial partitions is 13, this allows numerical calculations. The paper \cite{Redei-Jing2026} also formulates the following   

\medskip
\noindent {\bf Conjecture 7.3 in \cite{Redei-Jing2026}}:
For every $n\ge 3$ there exists a non-trivial decision context $\langle (X,\mathcal{S},p^*) f_1, f_2\rangle$ (with $X$ having $n$ elements) and a prior $p$ such that the decision
$\E_{p^*}[f_1]>\E_{p^*}[f_2]$ is conditionally $p$-inaccessible.
\medskip

We prove this conjecture by proving the even stronger Theorem \ref{thm:main1} below. Before stating the theorem we recall the notion of Bayes Blind Spot of a probability measure:
\begin{definition}[\cite{GyenisZ-RedeiBayes-General}, \cite{GyenisZ-Redei-Bayes-BlindSpot-Synthese}]
Given a probability measure space $(X,\mathcal{S}, p)$, the Bayes $p$-Blind Spot of $p$ is the set of probability measures on $\mathcal{S}$ that are absolutely continuous with respect to $p$ and which cannot be obtained by conditioning $p$ on evidence about them; i.e. which cannot be of the form $p_{\Pi}$ for any (non-trival proper) partition $\Pi$.   
\end{definition}
It was proved in  \cite{GyenisZ-Redei-Bayes-BlindSpot-Synthese} that the Bayes $p$-Blind Spot is a large set for all probability measures $p$ in probability spaces having a finite number of elementary events, and it is known that the same holds for certain probability spaces with an infinite number of elementary events  \cite{GyenisZ-RedeiBayes-General}, \cite{Shattuck-Wagner2024}. 
\begin{theorem}\label{thm:main1}
	For any given $p^*$ and $p$ on the measurable space $(X,\mathcal{S})$ with $X$ having $n\geq 3$ number of elements the following are equivalent:
	\begin{enumerate}
		\item[(A)] There exist real valued random variables (utility functions) $f_1$ and $f_2$ on $X$ such that the decision 
        \begin{equation}\label{eq:decision-in proof}
        \E_{p^*}[f_1]>\E_{p^*}[f_2]    
        \end{equation}  
        is conditionally $p$-inaccessible.
		\item[(B)] $p^*$ is in the Bayes $p$-Blind Spot (i.e. $q_\Pi\neq p^*$ for every proper non-trivial partition $\Pi$).
	\end{enumerate}	
\end{theorem}
\begin{proof}[of Theorem \ref{thm:main1}] The proof of this theorem will be based on proving two lemmas but the main idea of the proof is the following: It is clear that a necessary condition for a decision determined by $p^*$ to be conditionally $p$-inaccessible is that $p^*$ is in the Bayes $p$-Blind Spot. One can show that $p^*$ is in the Bayes $p$-Blind Spot if and only if the Radon-Nikodym derivative of $p^*$ with respect to $p$ is an injective function. Taking the logarithm of this Radon-Nikodym derivative allows utilizing the behavior of the Kullback-Leibler divergence between $p^*$ and $p$ to derive an inequality, which in turn entails that a function $d$ defined suitably in terms of the logarithm of the Radon-Nikodym derivative yields a conditionally $p$-inaccessible decision. 

Details: We write the decision (\ref{eq:decision-in proof}) in the equivalent form $\E_{p^*}[d]>0$ with $d=f_1-f_2$. 

\noindent
	Showing the implication (A)$\Rightarrow$(B): \\  If $q_\Pi=p^*$ for some proper non-trivial partition $\Pi$, then the requirement $\E_{q_\Pi}[d]\le 0$ forces $\E_{p^*}[d]\le 0$, contradicting $\E_{p^*}[d]>0$.
	
	\noindent 
    Showing the implication (B)$\Rightarrow$(A):\\ We split the proof of this implication into two lemmas. 
	First, define the function $r\colon X\to\mathbb{R}$ by
	\begin{equation}\label{eq:def-of-r}
		r(i)=\frac{p^*(i)}{p(i)} \quad \text{ for } i\in X.
	\end{equation}
	Note: the function ${r}$ is the Radon -Nikodym derivative of $p^*$ with respect to $p$.
	\begin{lemma}\label{lem:qequals}
		For a partition $\Pi$, one has $q_\Pi=p^*$ if and only if $r$ is constant on each block $B\in\Pi$.
	\end{lemma}
	\begin{proof}[of Lemma \ref{lem:qequals}]
		Lemma \ref{lem:qequals} was in fact proved in \cite{GyenisZ-Redei-Bayes-BlindSpot-Synthese} (see Proposition 3.1 in \cite{GyenisZ-Redei-Bayes-BlindSpot-Synthese}), where it was stated that $p^*$ is in the Bayes Blind Spot of $p$ if and only if the Radon-Nikodym derivative of $p^*$ with respect to $p$ is an injective function. For completeness we provide here a compact proof:  \\
        Fix $B\in\Pi$ and $i\in B$. The identity $q_\Pi(i)=p^*(i)$ is
		\[
		p^*(B)\frac{p(i)}{p(B)}=p^*(i)
		\quad\Longleftrightarrow\quad
		\frac{p^*(i)}{p(i)}=\frac{p^*(B)}{p(B)}.
		\]
		The right-hand side depends only on $B$, not on $i\in B$, hence $r$ must be constant on $B$.
		Conversely, if $r$ is constant on $B$ then the displayed equivalence holds for all $i\in B$, hence $q_\Pi=p^*$. End of proof of Lemma \ref{lem:qequals}
	\end{proof}

	Continuing the proof of Theorem \ref{thm:main1}: In particular, if $r$ takes between $2$ and $n-1$ distinct values, then letting $\Pi$ be the partition into
	level sets of $r$ yields $q_\Pi=p^*$. Also, if $p=p^*$ (i.e.\ $r$ is constant) then $q_\Pi=p^*$ for every $\Pi$.
	Thus, assumption (A) implies that $r$ is \emph{injective}.
	
	\noindent Let $g:X\to \mathbb{R}$ be an arbitrary function. By definition,
	\begin{equation}
	    \E_{p(\cdot\mid B)}[g] = \sum_{i}p(i\mid B)g(i)\,,
	\end{equation}
	and thus
	\begin{equation}\label{eq:Jeffrey-expect}
		\E_{q_\Pi}[g]=\sum_{B\in\Pi} p^*(B)\,\E_{p(\cdot\mid B)}[g]\,.
	\end{equation}
	By the law of total expectation\footnote{$\E_{p}(X) = \sum_{B\in\Pi}\E_{p}(X\mid B)p(B)$, where $\E_{p}(X\mid B) = \E_{p(\cdot\mid B)}(X)$.}
	for $p^*$,
	\begin{equation}\label{eq:seged:1}
		\E_{p^*}[g]=\sum_{B\in\Pi} p^*(B)\,\E_{p^*(\cdot\mid B)}[g].
	\end{equation}
	Subtractiing (\ref{eq:Jeffrey-expect}) from (\ref{eq:seged:1}) gives
	\begin{equation}\label{eq:loss-decomp}
		\E_{p^*}[g]-\E_{q_{\Pi}}[g]
			=\sum_{B\in\Pi} p^*(B)\Big(\E_{p^*(\cdot\mid B)}[g]-\E_{p(\cdot\mid B)}[g]\Big).
	\end{equation}
	
	\noindent Let us now define the function $g\colon X\to\mathbb{R}$ by
	\begin{equation}\label{def:g}
		g(i)=\left\{
  \begin{array}{ll}
    \log\frac{p^*(i)}{p(i)}, & \hbox{if\ } p(i)\not=0 \\
    0, & \hbox{if\ } p(i)=0
  \end{array}
\right.
	\end{equation}

	\begin{lemma}\label{lem:KL}
		If $r$ is injective, then for every proper non-trivial partition $\Pi$ we have
		\[
			\E_{q_{\Pi}}[g] < \E_{p^*}[g]\,.
		\]
	\end{lemma}
	\begin{proof}[of Lemma \ref{lem:KL}]
		Recall (see e.g. \cite{Amari2016}[p. 57]) that the Kullback--Leibler divergence between two probability measures $q_1$ and $q_2$ is defined by
        \begin{equation}
         D(q_1\| q_2)=\sum_i q_1(i)\log\frac{q_1(i)}{q_2(i)}
  \end{equation}
        with the convention that $ q_1(i)\log\frac{q_1(i)}{q_2(i)}$ is taken to be 0 if $q_1(i)=0$ and $+\infty$ if $q_1(i)>0$ and $q_2(i)=0$ for some $i$. The divergence
        $D(q_1\| q_2)$ is always non-negative \cite{Amari2016}[p. 59] (Gibbs' inequality), and it is finite if $q_2(i)>0$ for all $i$. Taking $q_1=p^*$ and $q_2=p$ and considering (\ref{def:g}) we obtain
		\begin{equation}\label{eq:Div}
			D(p^*\|p)=\sum_{i\in X} p^*(i)\log\frac{p^*(i)}{p(i)}=\E_{p^*}[g].
		\end{equation}
        Since we assumed that $p(i)>0$ for all $i$, the Kullback--Leibler divergence $D(p^*\|p)$ is finite. 
		To simplify formulas we introduce the abbreviations $p^*_B=p^*(\cdot\mid B)$ and $p_B=p(\cdot\mid B)$. For a fixed block $B$ and $i\in B$ we have
		\begin{equation}\label{g-on-B}
		  g(i)=\log\frac{p^*(i)}{p(i)}
			=\log\frac{p^*(i\mid B)}{p(i\mid B)}+\log\frac{p^*(B)}{p(B)}.  
		\end{equation}
		Taking the 
        the expectation value of $g$ with respect to $p^*_B$ and $p_B$, using (\ref{g-on-B}) and keeping mind the convention that the contribution to the Kullback-Leibler divergence of a term $ q_1(i)\log\frac{q_1(i)}{q_2(i)}$ is taken to be 0 if $q_1(i)=0$, we obtain
		\begin{eqnarray}
			\E_{p^*_B}[g]&=&D(p^*_B\|p_B)+\log\frac{p^*(B)}{p(B)}\label{eq:help1}\\
			\E_{p_B}[g]&=&-D(p_B\|p^*_B)+\log\frac{p^*(B)}{p(B)}\label{eq:help2}.
		\end{eqnarray}
		Subtracting (\ref{eq:help2}) from (\ref{eq:help1}) cancels the $\log\frac{p^*(B)}{p(B)}$ term and yields
		\begin{equation}
			\E_{p^*_B}[g]-\E_{p_B}[g]=D(p^*_B\|p_B)+D(p_B\|p^*_B)\ge 0, \label{eq:kulonbseg}
		\end{equation}
		with equality iff $p^*_B=p_B$. Summing over blocks with weights $p^*(B)$, \eqref{eq:loss-decomp} gives
		\[
			\E_{q_{\Pi}}[g] \leq \E_{p^*}[g]\,.
		\]
		Equality $\E_{q_\Pi}[g]=D(p^*\|p)$ holds iff $p^*(\cdot\mid B)=p(\cdot\mid B)$ for every
		block $B\in\Pi$. By Lemma~\ref{lem:qequals}, this is equivalent to $r$ being constant on each block.
		If $r$ is injective, the only subsets on which $r$ is constant are singletons. Every proper non-trivial partition has at least
		one block $B$ with $|B|\ge 2$, so on that block $r$ is not constant; hence $p^*(\cdot\mid B)\neq p(\cdot\mid B)$, so the
		corresponding term in \eqref{eq:kulonbseg} is strictly positive. Therefore the whole sum is
		positive and $\E_{q_\Pi}[g] < \E_{p^*}[g]$. End of proof of Lemma \ref{lem:KL}.
	\end{proof}
	
	Continuing the proof of Theorem \ref{thm:main1}: By assumption (B) we already know that $r$ is injective; hence by Lemma \ref{lem:KL} for every proper non-trivial partition
	$\Pi$ we have $\E_{q_{\Pi}}[g] < \E_{p^*}[g]$.
	Set
	\[
		M=\max_{\Pi\in\mathfrak P} \E_{q_\Pi}[g].
	\]
	Finiteness of $\mathfrak P$ ensures that the maximum exists.
	By Lemma \ref{lem:KL}, for every proper non-trivial $\Pi$ we have $\E_{q_\Pi}[g]<\E_{p^*}[g]$, hence
	\[
		M < \E_{p^*}[g].
	\]
	Let $\Delta =\E_{p^*}[g]-M>0$ and fix any $\varepsilon$ with $0<\varepsilon<\Delta$. Define the function $d\colon X\to \mathbb{R}$ by
	\[
		d(i) = g(i)-(M+\varepsilon),
	\]
	and set $f_2\equiv 0$ and $f_1=d$. Then we have the decision
    \begin{equation}\label{eq:decision-in-proof}
     \E_{p^*}[d]=\E_{p^*}[g]-(M+\varepsilon)=\Delta-\varepsilon>0\,.  
     \end{equation}
	On the other hand, for every proper non-trivial $\Pi$ we have
	\[
		\E_{q_\Pi}[d]=\E_{q_\Pi}[g]-(M+\varepsilon)\le M-(M+\varepsilon)=-\varepsilon<0,
	\]
	so the decision (\ref{eq:decision-in-proof}) is \emph{strongly} conditionally $p$-inaccessible. 
	This establishes (A).
\end{proof}
It is known (see \cite{GyenisZ-Redei-Bayes-BlindSpot-Synthese}) that the Bayes $p$-Blind Spot of every probability measure $p$ is a very large set: it has continuum cardinality, it is topologically fat (Bair second category) and it has the same measure as the measure of the set of all probability measures (with respect to a natural measure on the set of all probability measures). In view of this, Theorem \ref{thm:main1} says that given \textit{any} subjective probability $p$,  there exist a \textit{continuum} number of objective probability measures for \textit{each} of which there exist a \textit{continuum} number of pairs of utility functions (parametrized by the real number $\varepsilon$ in the proof) such that the objectively good decision based on calculating their expectation values using the objective probability is conditionally $p$-inaccessible. Thus, conditional $p$-inaccessibility is not a rare phenomenon; quite on the contrary: it is a robust and substantial one. For any given subjective probability the set of potential objective probabilities that determine objectively good decisions that are conditionally $p$-inaccessible dominate the set of all probability measures -- this is because the Bayes $p$-Blind Spot does. 
Thus there is an enormous number of decision situations in which the Agent's subjective probability prohibits the Agent's informed rational decision to be objectively good.         

\def\Bell{\operatorname{Bell}}

\section{Degree of conditional $p$-inaccessibility}
In \cite{Redei-Jing2026} a notion of degree of condional $p$-inaccessibility of a decision was defined, which is intended to characterize how far a decision is from being conditionally $p$-inaccessible (Definition 7.1 in \cite{Redei-Jing2026}). In the notation of this paper the definition is: 

\begin{definition}[Degree of conditional $p$-inaccessibility] Given a decision context $\langle (X,\mathcal{S},p^*) f_1, f_2\rangle$ and a subjective probability $p$, for the decision $\E_{p^*}[d]>0$ with $d=f_1,-f_2$, define the \emph{inaccessible set} $\mathcal{I}(d)$ by
	\[
		\mathcal I(d)=\{\Pi\in\mathfrak P:\ \E_{q_\Pi}[d]\le 0\}.
	\]
	The \emph{degree} of condional $p$-inaccessibility (denoted by $\deg(d)$) of the decision $\E_{p^*}[d]>0$ is by definition the size of $\mathcal{I}(d)$:
	\[
		\deg(d)=|\mathcal I(d)|.
	\]
    where $|\mathcal{I}(d)|$ is the cardinality of the set $\mathcal{I}(d)$.
\end{definition}
Thus if $\deg(d)$ is equal to the number of all proper non-trivial partitions of the set $X$ having $n$ elements, then the decision is conditionally $p$-inaccessible. The number of all partitions of a set having $n$ elements is called the Bell-number $\Bell(n)$ \cite{Conway-Guy1996}; so the number of all proper non trivial partitions is $\Bell(n)-2$. In case of a conditional $p$-inaccessible decision we thus have $\deg(d)=\Bell(n)-2$ because the finest partition is not proper and the two element partition is trivial. If $\deg(d)=0$, then $p$ is such that updating it on \emph{any} partial information about $p^*$ and using the updated probability measure to calculate expectation values of the utility functions, we obtain the objectively good decision. In the intermediate cases the higher the number $\deg(d)$ the less suitable $p$ is for making decisions in the decision context determined by the objective probability $p^*$.    
\medskip

\noindent {\bf Conjecture 7.2 in \cite{Redei-Jing2026}:}
In any finite probability space having $n \geq 3$ number of elementary events there exist decisions 
such that for any $k$ such that $0 \leq k \leq \Bell(n) - 2$ there exist priors $p_k$ with the property that the decisions are conditionally $p_k$-inaccessible to degree $k$.

\medskip
Here we prove the above conjecture. 
The conjecture can be re-phrased by saying that the achievable degrees of conditional $p$-inaccessibility for a decision is the full set  $k\in\{0,1,\dots, \Bell(n)-2\}$, where the notion of ``achievable degrees of conditional $p$-inaccessibility of a decision'' is given by the following definition:

\begin{definition}
  Given a decision context $\big\langle (X,\mathcal{S},{p}^*),f_1,f_2\big\rangle$, we call the  numbers $k\in\{0,1,\dots, \Bell(n)-2\}$ for which there exist $p_k$ such that the decision $\E_{p^*}[d]>0$ ($d=f_1-f_2)$) is conditionally $p_k$-inaccessible to degree $\deg(d)=k$ the \emph{achievable degrees of condional} $p$-\emph{inaccessibility} of the decision.    
\end{definition}

The obstruction beyond $p^*$ being in the Bayes $p$-Blind spot is that different partitions can, in principle, yield the same posterior.

\begin{definition}
	Let
	\[
		\mathcal Q=\{q_\Pi:\ \Pi\in\mathfrak P\}
	\]
	be the set of posteriors. For $q\in\mathcal Q$ define its multiplicity
	\[
		m(q)=\big|\{\Pi\in\mathfrak P:\ q_\Pi=q\}\big|.
	\]
	Thus $\sum_{q\in\mathcal Q} m(q)=\Bell(n)-2$.
\end{definition}
The next theorem characterizes the achievable degrees of conditional $p$-inaccesibility for fixed $(p^*,p)$:
\begin{theorem}\label{thm:degrees}
	Consider again the function $r$ defined by eq. (\ref{eq:def-of-r}) and assume that it is injective. Then there exists a single function $u:X\to\mathbb R$ such that the numbers
	$\{\E_q[u]:q\in\mathcal Q\}$ are pairwise distinct. Fix such a function $u$, let $g$ be the function specified by \eqref{def:g} and set $g_\eta=g+\eta u$
	for sufficiently small $\eta>0$.
	Then:
	\begin{enumerate}
		\item All values $\E_{q_\Pi}[g_\eta]$ depend only on $q_\Pi$ and are strictly smaller than $\E_{p^*}[g_\eta]$.
		\item If we list $\mathcal Q=\{q^{(1)},\dots,q^{(L)}\}$ so that
		\[
			\E_{q^{(1)}}[g_\eta]<\E_{q^{(2)}}[g_\eta]<\cdots<\E_{q^{(L)}}[g_\eta],
		\]
		and define the cumulative sums
		\[
			K_\ell=\sum_{j=1}^{\ell} m\big(q^{(j)}\big)\qquad(\ell=1,\dots,L),
		\]
		then the set of degrees realized by \emph{some} $d$ with $\E_{p^*}[d]>0$ is precisely
		\[
			\{0,\,K_1,\,K_2,\,\dots,\,K_L\}=\Big\{\sum_{j=1}^{\ell} m\big(q^{(j)}\big):\ \ell = 1, \ldots, L\Big\}.
		\]
		Moreover, each such degree is realized by a function of the form $d=g_\eta-c$ with a suitable constant $c$.
	\end{enumerate}
\end{theorem}

\begin{proof} We prove the theorem in three steps:\\
	\textbf{Step 1:} The existence of $u$.
	Think of each $q\in\mathcal Q$ as a vector in $\mathbb R^n$.
	For distinct $q\neq q'$, the set of $u\in\mathbb R^n$ with $\E_q[u]=\E_{q'}[u]$ is the hyperplane
	$\{u:\ (q-q')\cdot u=0\}$. Since there are finitely many pairs, the union of these hyperplanes is a proper subset of $\mathbb R^n$,
	so we may choose $u$ outside of it. Then $\E_q[u]$ are pairwise distinct for $q\in\mathcal Q$.

	\noindent \textbf{Step 2}: Let $g$ be as in \eqref{def:g}, and $g_\eta=g+\eta u$ for sufficiently small $\eta>0$.
	By the proof of Theorem \ref{thm:main1},
	\[
		\delta=\E_{p^*}[g]-\max_{\Pi\in\mathfrak P}\E_{q_\Pi}[g] > 0.
	\]
	Let $R$ be defined by 
    \begin{equation}
    R=\max_{q\in\mathcal Q}|\E_q[u]|+|\E_{p^*}[u]|<\infty
    \end{equation}.
	If $0<\eta<\delta/(2R)$, then
	\[
	\E_{p^*}[g_\eta]-\max_{\Pi}\E_{q_\Pi}[g_\eta]
	\ge \delta-\eta\cdot 2R>0,
	\]
	so $\E_{q_\Pi}[g_\eta]<\E_{p^*}[g_\eta]$ for all $\Pi$.

	\noindent \textbf{Step 3}:
	Because $\E_q[u]$ are pairwise distinct on $\mathcal Q$, for small $\eta$ the values
	$\E_q[g_\eta]=\E_q[g]+\eta\E_q[u]$ are also pairwise distinct on $\mathcal Q$.
	Order $\mathcal Q=\{q^{(1)},\dots,q^{(L)}\}$ increasingly by $\E_{q^{(j)}}[g_\eta]$.
	
	Fix $\ell\in\{0,1,\dots,L\}$, where $\ell=0$ means ``none''.
	Choose a constant $c$ such that
	\[
	\E_{q^{(\ell)}}[g_\eta] < c < \E_{q^{(\ell+1)}}[g_\eta],
	\]
	with the conventions $\E_{q^{(0)}}[g_\eta]=-\infty$ and $\E_{q^{(L+1)}}[g_\eta]=+\infty$.
	Define $d=g_\eta-c$. Then for a partition $\Pi$,
	\[
	\E_{q_\Pi}[d]\le 0
	\quad\Longleftrightarrow\quad
	\E_{q_\Pi}[g_\eta]\le c
	\quad\Longleftrightarrow\quad
	q_\Pi\in\{q^{(1)},\dots,q^{(\ell)}\}.
	\]
	Hence, exactly $\sum_{j\le \ell} m(q^{(j)})=K_\ell$ partitions lie in $\mathcal I(d)$, i.e.\ $\deg(d)=K_\ell$.

	Finally, since 
    \begin{equation}
    c<\E_{q^{(\ell+1)}}[g_\eta]\le \max_{\Pi}\E_{q_\Pi}[g_\eta]<\E_{p^*}[g_\eta]   
    \end{equation} we also have
	\begin{equation}
	 \E_{p^*}[d]=\E_{p^*}[g_\eta]-c>0.   
	\end{equation}
	Thus each $K_\ell$ is achievable, and no other values are possible because any $d$ induces an ordering of $\mathcal Q$
	and can only select unions of initial segments, counted with multiplicity.
\end{proof}

\begin{corollary}\label{cor:full}
	Assume $r$ is injective. Then the following are equivalent:
	\begin{enumerate}
		\item For every $k\in\{0,1,\dots,\Bell(n)-2\}$ there exists $d$ with $\E_{p^*}[d]>0$ and $\deg(d)=k$.
		\item The map $\Pi\mapsto q_\Pi$ is injective on $\mathfrak P$ (equivalently, $m(q)=1$ for all $q\in\mathcal Q$).
	\end{enumerate}
	Under these equivalent conditions, one may realize every $k$ by thresholding $d=g_\eta-c$ as in Theorem~\ref{thm:degrees}.
\end{corollary}

\begin{proof}
	By Theorem~\ref{thm:degrees}, the achievable degrees are exactly $\{0,K_1,\dots,K_L\}$.
	This set equals $\{0,1,\dots,M\}$ iff all multiplicities are $1$, i.e.\ iff $\Pi\mapsto q_\Pi$ is injective.
\end{proof}

\section*{Appendix}

\begin{proof}[of Theorem \ref{thm:main2}] 
The proof consists of two parts: in Part 1 we assume that $p^*$ is nowhere zero, and we prove that the theorem's claim is true under this assumption. In the second part we allow $p^*$ to take zero value on some $i$. In this second part, using $p^*$, $p$ and $d$, we construct a new decision context $\big\langle (X,\mathcal{S},p^{\varepsilon}),f^{\varepsilon}_1,f^{\varepsilon}_2\big\rangle$ with decision $\E_{p^\varepsilon}[d^{\varepsilon}]>0$ ($d^{\varepsilon}=f^{\varepsilon}_1-f^{\varepsilon}_2$ ), with a \emph{nowhere zero} $p^{\varepsilon}$. Thus, by Part 1 of the proof, for this new decision the theorem's claim holds, and we show that this is in contradiction with the assumption that $\E_p[d]\geq 0$.  \\ 

{\bf Part 1:}
    Assume that $p^*$ is nowhere zero, and let 
    $r(i) = p(i)/p^*(i)$. 
	By Theorem \ref{thm:main1} we know that $r$ is injective (in Theorem \ref{thm:main1} injectivity is proved for the function $i\mapsto p^*(i)/p(i)$,
    and thus its reciprocal function is also injective). Hence we may order $X$ so that 
	\[
		r(1)<r(2)<\cdots<r(n).
	\]
	For each $1\leq m \leq n-1$, let $\Pi_m$ be the proper partition that has exactly one $2$-block
	$\{m,m+1\}$ and otherwise singletons. Let $q_m=q_{\Pi_m}$. Then
	\[
		q_m(i) = \begin{cases}
			p^*(i) & \text{ if } i\notin\{m,m+1\},\\
			\frac{p^*(\{m,m+1\})}{p(\{m, m+1\})} p(i) & \text{ if } i=m, \text{ or } i=m+1\,.			
		\end{cases}
	\]
	For $1\leq m\leq n-1$ let
	\[
		S_m =\sum_{i\le m}\bigl(p^*(i)-p(i)\bigr) = p^*(\{1, \ldots, m\}) - p(\{1, \ldots, m\})
	\]
	We claim that $S_m>0$ for every $m$. Indeed, as $r$ is strictly increasing there is an index $k$ where $r$ crosses $1$.
	Then, for $m\leq k$ where $r(m)\leq 1$, from the definition we have $S_m > 0$. For $m\geq k$ where $r(m)\geq 1$,
	we have $S_m=\sum_{i>m}(p(i)-p^*(i))$  (because $\sum_i p^*(i)r(i)=\sum_i p(i)=1$), and thus $S_m>0$ follows.

	Consider $p$, $p^*$, and $q_m$ as vectors in $\mathbb{R}^n$, and let $e_k$ be the $k$th standard basis vector in $\mathbb{R}^n$.
	By an elementary telescoping it is easy to check that
	\begin{equation}
		p-p^* = \sum_{m=1}^{n-1}S_m(e_{m+1}-e_m)\,. \label{eq:pminuspstar}
	\end{equation}
	\begin{lemma}\label{seged}
		There are scalars $t_m>0$ such that
		\[
			p-p^* = \sum_{m=1}^{n-1}t_m(q_m-p^*)\,,
		\]
		and $t = \sum_{m=1}^{n-1}t_m > 1$.
	\end{lemma}
	\begin{proof}[of Lemma \ref{seged}]
		By inspecting $q_m-p^*$ the idea is to define
		\[
			A_m = q_m(m+1)-p^*(m+1)=\frac{p(m+1)p^*(m)-p(m)p^*(m+1)}{p(m)+p(m+1)}.
		\]
		Since $r(m)<r(m+1)$, the numerator is positive, hence $A_m>0$.
		Moreover
		\begin{equation}\label{eq:qm-shift}
			q_m-p^* \;=\; A_m\,(e_{m+1}-e_m),
		\end{equation}
		Combining this with \eqref{eq:pminuspstar} we get
		\[
			p-p^* = \sum_{m=1}^{n-1}\frac{S_m}{A_m}(q_m-p^*)\,
		\]
		and thus letting $t_m = \frac{S_m}{A_m}$ we need to prove $t_m>0$ and $t = \sum_{m=1}^{n-1}t_m > 1$.
		We in fact prove that $t_m\geq 1$, and as $n$ is at least $3$, the result follows.\\
		
		\noindent Claim: For each $m$, one has $A_m\le S_m$, and consequently $\frac{S_m}{A_m}\ge 1$.
		Write
		\[
			a =p^*(m),\quad  b=p^*(m+1),\quad u=p(m), \quad v=p(m+1)\,.
		\]
		So $A_m=\frac{av-ub}{u+v}>0$. There are two cases.\\

		\noindent Case (i): $u+v\le a+b$. Then necessarily $u\le a$ and $v\le b$
		(otherwise $u>a$ would imply $r(m)>1$ and hence	$r(m+1)>1$, giving $v>b$ and thus $u+v>a+b$).
		In particular, every ratio $r(i)\le r(m)\le 1$ for $i\le m$, so each $p^*(i)-p(i)\ge 0$ and hence $S_m\ge a-u$.
		Also,
		\[
			A_m\le a-u
			\quad\Longleftrightarrow\quad
			av-ub \le (a-u)(u+v)
			\quad\Longleftrightarrow\quad
			b \ge (u+v)-a,
		\]
		and the last inequality holds because $u+v\le a+b$. Thus $A_m\le a-u\le S_m$.

		\noindent Case (ii): $u+v\ge a+b$. Then necessarily $v\ge b$ (otherwise $v<b$ would imply $r(m+1)<1$, hence $r(m)<1$, giving $u<a$ and
		thus $u+v<a+b$). Hence $v-b\ge 0$, and since every $r(i)\ge r(m+1)\ge 1$ for $i>m$, we have
		\[
		S_m=\sum_{i>m}\bigl(p(i)-p^*(i)\bigr)\ge v-b.
		\]
		Moreover,
		\[
		A_m\le v-b
		\quad\Longleftrightarrow\quad
		av-ub \le (v-b)(u+v)
		\quad\Longleftrightarrow\quad
		0 \le v\bigl((u+v)-(a+b)\bigr),
		\]
		which holds because $u+v\ge a+b$. Thus $A_m\le v-b\le S_m$.
	\end{proof}
	
	\noindent Rearranging the equation in Lemma \ref{seged} we get
	\[
		p = \sum_{m=1}^{n-1}t_m q_m - (t-1) p^*\,.
	\]
	Taking expected values,
	\[
		\E_{p}[d] = \sum_{m=1}^{n-1}t_m \E_{q_m}[d] - (t-1) \E_{p^*}[d]\,.
	\]
	By the assumption of the theorem ($p$-inaccessibility of the decision $\E_{p^*}(d)$), 
    and since each $\Pi_m$
	is proper non-trivial, we have $\E_{q_m}[d]\leq 0$ and $\E_{p^*}[d]>0$, and since $t-1>0$, we also
	have $(t-1)\E_{p^*}[d] > 0$. It follows that
	\[
		\E_{p}[d] \leq 0 - (t-1)\E_{p^*}[d] < 0\,.
	\]
    
    {\bf Part 2:} 
    Now that the theorem has already been proved in the case where $p^*$ is strictly positive at every point, let us deal with the general case when $p^*$ can have zero values: assume  $p^*(i)=0$ for some $i\in X$. (Recall we assumed that $p(i)>0$ for every $i\in X$.) Suppose that $d:X\to\mathbb R$ satisfies
    \[
        \E_{p^*}[d]>0\qquad\text{and}\qquad
        \E_{q_\Pi}[d]\leq 0
    \]
    for every proper non-trivial partition $\Pi$ of $X$. We have to prove that $\E_p[d]<0$. 
    
    For $0<\varepsilon<1$, define a new probability measure $p^\varepsilon$ by
    \[
        p^\varepsilon=(1-\varepsilon)p^*+\varepsilon p.
    \]
    Since $p(i)>0$ for every $i\in X$, we have $p^\varepsilon(i)>0$ for every $i\in X$. Thus $p^\varepsilon$ is nowhere zero.

    Let $q^\varepsilon_\Pi$ denote the Jeffrey posterior obtained from the prior $p$ determined by $\Pi$ by conditioning using $p^\varepsilon$, that is,
    \[
        q^\varepsilon_\Pi(i) = \sum_{B\in \Pi}
        p^\varepsilon(B)p(i\mid B).    
    \]
    We claim that for every partition $\Pi$,
    \[
        q^\varepsilon_\Pi=(1-\varepsilon)q_\Pi+\varepsilon p.
    \]
    Indeed, if $i\in B\in\Pi$, then
    \begin{align*}
        q^\varepsilon_\Pi(i) &=
        p^\varepsilon(B)\frac{p(i)}{p(B)} =
        \big((1-\varepsilon)p^*(B)+\varepsilon p(B)\big)\frac{p(i)}{p(B)} \\ 
        &= (1-\varepsilon)p^*(B)\frac{p(i)}{p(B)} +
        \varepsilon p(i) \\
        &= (1-\varepsilon)q_\Pi(i)+\varepsilon p(i).      
    \end{align*} 
    Assume, in order to derive a contradiction, that $\E_p[d]\geq 0$. Define
    \[
        d^\varepsilon=d-\varepsilon \E_p[d].
    \]
    Then
    \begin{align*}
        \E_{p^\varepsilon}[d^\varepsilon] &=
        \E_{p^\varepsilon}[d-\varepsilon\E_p[d]] = 
        \E_{p^\varepsilon}[d] -
        \E_{p^\varepsilon}[\varepsilon\E_p[d]] =
        \E_{p^\varepsilon}[d] -
        \varepsilon\E_p[d]  \\
        &=
        (1-\varepsilon)\E_{p^*}[d] +
        \varepsilon \E_p[d] -
        \varepsilon \E_p[d] =
        (1-\varepsilon)\E_{p^*}[d].
    \end{align*}
    Since $\E_{p^*}[d]>0$, it follows that $\E_{p^\varepsilon}[d^\varepsilon]>0$.
    Let $\Pi$ be any proper non-trivial partition. Using 
    $q^\varepsilon_\Pi=(1-\varepsilon)q_\Pi+\varepsilon p$, we obtain
    \[
        \E_{q^\varepsilon_\Pi}[d^\varepsilon] =
        (1-\varepsilon)\E_{q_\Pi}[d] +
        \varepsilon \E_p[d] -
        \varepsilon \E_p[d] =
        (1-\varepsilon)\E_{q_\Pi}[d].
    \]
    Since $\E_{q_\Pi}[d]\leq 0$, this gives $\E_{q^\varepsilon_\Pi}[d^\varepsilon]\leq 0$ for every proper non-trivial partition $\Pi$.

    Thus, $d^\varepsilon$ satisfies the hypotheses of the theorem with $p^\varepsilon$ in place of $p^*$. Since $p^\varepsilon$ is nowhere zero, 
    the proof in Part 1 about the strictly positive $p^*$ yields
    \[
        \E_p[d^\varepsilon]<0.
    \]
    But
    \[
        \E_p[d^\varepsilon] =
        \E_p[d-\varepsilon \E_p[d]] =
        (1-\varepsilon)\E_p[d].
    \]
    Therefore $(1-\varepsilon)\E_p[d]<0$. Since $1-\varepsilon>0$, we get $\E_p[d]<0$, contradicting the assumption $\E_p[d]\geq 0$. 
    Consequently, $\E_p[d]<0$. 
\end{proof}

\section*{Acknowledgment}
Research supported by the project no. 2022/47/B/HS1/01581 of the National Science Centre, Poland; and by the Hungarian National Research, Development and Innovation Office, grant number: ADVANCED 152165.


\end{document}